\documentclass{article}
\usepackage{mathtools,amsfonts,amssymb,amsthm,amscd}
\usepackage{mathrsfs}
\usepackage{color}
\usepackage{comment}
\usepackage{tikz}
\usepackage{autobreak}
\usepackage{float}
\usepackage{hyperref}
\usepackage{cleveref}
\usepackage{autonum}

\crefname{theorem}{Theorem}{Theorems}
\crefname{lemma}{Lemma}{Lemmas}
\crefname{proposition}{Proposition}{Propositions}
\crefname{corollary}{Corollary}{Corollaries}
\crefname{definition}{Definition}{Definitions}
\crefname{example}{Example}{Examples}
\crefname{remark}{Remark}{Remarks}
\crefname{section}{Section}{Sections}
\crefname{figure}{Figure}{Figures}
\crefname{equation}{}{}
\crefname{enumi}{}{}

\usetikzlibrary{positioning}
\newcommand{\field}[1]{\mathbb{#1}}
\newcommand{\Z}{\field{Z}}
\newcommand{\Q}{\field{Q}}
\newcommand{\R}{\field{R}}

\newcommand{\F}{\field{F}}
\newcommand{\Oint}{\mathcal{O}}

\newcommand{\iK}{\mathrm{i}_{K}}
\newcommand{\iphi}{\mathrm{i}_{\varphi}}
\newcommand{\Gal}{\mathrm{Gal}}

\newtheorem{theorem}{Theorem}[section]
\newtheorem{lemma}[theorem]{Lemma}
\newtheorem{proposition}[theorem]{Proposition}

\theoremstyle{definition}

\newtheorem{definition}[theorem]{Definition}

\theoremstyle{remark}
\newtheorem{remark}[theorem]{Remark}

\allowdisplaybreaks
\numberwithin{equation}{section}
\tikzset{sgplattice/.style={inner sep=1pt,norm/.style={red!50!blue},char/.style={blue!50!black},
  lin/.style={black!50}},cnj/.style={black!50,yshift=-2.5pt,left=-1pt of #1,scale=0.5,fill=white}}

\begin{document}

\title{Computating decomposition groups and inertia groups using Newton polygons}
%unramified over quadratic subfields}
\author{Kazuma Igarashi and Nozomu Suzuki}
\date{\today}
\maketitle
\begin{abstract}
   Newton polygons are powerful tools for computing the decomposition of prime ideals in extension rings.
  Methods based on Newton polygons have been developed through the work of Ore, Montes, and Nart.
  K\"{o}lle and Schmid obtained a method for determining decomposition groups from Newton polygons and associated data under assumptions introduced by Ore.
  In this paper, we extend their approach and show how to determine decomposition groups under the weaker assumptions introduced by Montes and Nart, formulated in terms of indices.

  \renewcommand{\thefootnote}{\fnsymbol{footnote}}
  \renewcommand{\thefootnote}{\arabic{footnote}}
\end{abstract}

\section{Introduction}\label{sec:Intro}

  In general, describing the prime ideal decomposition of a given extension of number fields is a highly difficult and fundamental problem in number theory.
  In 1878, Dedekind, in his paper \cite{Dedekind1878}, proved that under certain conditions, the prime ideal decomposition corresponds to the factorization of the defining polynomial into irreducible factors over the residue field.
  Dedekind's theorem is well known to be remarkable both theoretically and computationally, although its applicability is limited.

  In 1928, Ore extended Dedekind's approach by introducing the Newton polygon method, which enables the determination of the prime ideal decomposition under a weaker condition, called regularity.
  Ore's theorem applies in far more cases than Dedekind's theorem and describes the decomposition in terms of Newton polygons. 

  Furthermore, in 1992, Montes and Nart refined Ore's method and succeeded in computing prime ideal decompositions under an even weaker condition.
  The condition is necessary and sufficient for the applicability of their method. 

  On the other hand, Dedekind's theorem has also inspired further developments in the study of Galois groups.
  Beckmann (in \cite{MR1271245}) observed that, under the assumptions of Dedekind's theorem, the cycle type of a generator of the inertia group of a tamely ramified prime can be determined.
  As she was unable to locate a proof of this statement, she provided one in her paper.

  Furthermore, in 2004-2009, K\"{o}lle and Schmid extended these ideas by establishing an analogue of the above theorem under the assumption of Ore.
  Their method, which makes substantial use of local field theory, not only enabled the computation of decomposition and inertia groups but also led to further developments in this area.

  In this paper, we aim to establish an analogue of the theorem of K\"{o}lle and Schmid under the assumption of Montes and Nart (\cref{thm:our main theorem,thm:our main theorem 2}).
  In particular, we refine the polynomial construction of K\"{o}lle and Schmid, which is based on Ore's method, within the framework of Montes and Nart. 

\section{Preliminaries}\label{sec:Prelim}

  Let $p$ be a prime number. Fix an algebraic closure $\overline{\Q}_p$ (resp. $\overline{\field{F}}_p$) of $\field{Q}_p$ (resp. $\field{F}_p$). 
  Throughout, every algebraic extension of $\field{Q}_p$ (resp. $\field{F}_p$) is 
  regarded as a subfield of $\overline{\field{Q}}_p$ (resp. $\overline{\field{F}}_p$). 
  Let $K$ be a finite extension of $\field{Q}_p$. We denote by $\mathcal{O}_K$ the ring of integers of $K$, 
  by $\mathfrak{p}_K$ the prime ideal of $\mathcal{O}_K$, by $\field{F}_K$ the residue field, and by $v_K$ the valuation of $\overline{\Q}_p$ 
  normalized by $v_K(K^{\times}) = \Z$. 
  We fix a uniformizer $\pi \in \mathcal{O}_K$. 
  For $a \in \mathcal{O}_K$, the image of $a$ in $\field{F}_K$ is denoted by $\overline{a}$. 
  We shall denote the image of $\varphi(X) \in \mathcal{O}_K [X]$ in $\field{F}_K [X]$ also by $\overline{\varphi}(X)$. 
  In this paper, we assume that a polynomial $f$ is separable.

  \subsection{Newton polygons}\label{sec:Newton polygons}
    
    We define $\varphi$-polygons, which generalize the classical Newton polygon, and the polynomials associated to them. 
    To this end, we extend the valuation $v_K$ to $\mathcal{O}_K[X]$, which we denote by the same symbol, by 
    \begin{align}
      v_K:\mathcal{O}_K[X] & \rightarrow \Z_{\ge 0} \cup \{ \infty \}, \\
      a_0 + a_1 X + \cdots + a_n X^n & \mapsto \min\{v_K(a_i):0 \le i \le n \}. 
    \end{align}
    Let $\varphi(X) \in \mathcal{O}_K[X]$ be a monic irreducible polynomial of degree $l \ge 1$ such that 
    its reduction $\overline{\varphi}(X)$ is also irreducible over $\field{F}_K$. 
    Fix a root $\zeta$ of $\overline{\varphi}$. 
    Given a polynomial $f(X) \in \mathcal{O}_K[X]$ of degree $n$, we obtain a unique \emph{$\varphi$-expansion} as 
    \[
      f(X) = \sum_{i = 0}^{[n/l]} a_i(X) \varphi(X)^i, 
    \]
    where each $a_i(X) \in \mathcal{O}_K[X]$ is a polynomial of degree $< l$ or $a_i(X)=0$. 
    Throughout, we assume $a_0(X) \ne 0$. 

    When $f(X) \in \mathcal{O}_K[X]$ is expanded as above, we define the \emph{$\varphi$-polygon} of $f(X)$ 
    as the lower convex envelope in the Euclidean plane of points $\{(i,v_K(a_i))|0 \le i \le [n/l]\}$. 
    When $\varphi(X) = X$, the $\varphi$-polygon is the classical Newton polygon. 
    If $f(X)$ is a monic polynomial, we call the set of sides with negative slope the \emph{principal part} of $\varphi$-polygon. 
    When $\overline{\varphi}$ does not divide $\overline{f}$, the $\varphi$-polygon of $f$ has no principal part. 
    We say that the $\varphi$-polygon is \emph{one-sided} when it has exactly one side; otherwise, we say that it is \emph{many-sided}. 

    In general, let $S$  be a line segment in $\R^2$ with negative slope and its endpoints
    having integer coordinates. 
    Let its initial point be $(r,s)$ and its terminal point be $(r + E,s - H)$. 
    We call $E$ the \emph{length} of $S$ and $H$ the \emph{height} of $S$, denoted $E(S)$ and $H(S)$, respectively. 
    We define 
    \[
      d = \gcd(E(S),H(S)), 
    \]
    and call $d$ the degree of $S$. If $H = E = 0$, we set $d = 0$. 

    Let $f(X) \in \mathcal{O}_K[X]$ be a monic polynomial, and $S$ be a line segment 
    in $\{ (x,y) \in \R^2 \mid x \ge 0, y \ge 0 \}$ with negative slope and its endpoints having integer coordinates. 
    Suppose that no vertex of the $\varphi$-polygon of $f$ lies below $S$. 
    We denote the endpoints of $S$ by $P = (e_0,h_0)$ and $Q = (e_1,h_1)$, where $e_0 \le e_1$. 
    Write $d$ for the degree of $S$. 
    Put 
    \[
      b_j(X) = \pi^{-h_0 + jh} a_{e_0 + je}, 
    \]
    where $h = \frac{h_0 - h_1}{d}, e = \frac{e_1 - e_0}{d}$. 
    We define the \emph{polynomial associated to $f$ and $S$}, or the \emph{associated polynomial of $f$ and $S$} 
    as 
    \[
      f_S(Y) = \sum_{j=0}^d \overline{b_j}(\zeta) Y^j.
    \]
    If no vertex lies on $S$, we set $f_S(Y) = 0$. 
    We also define the normalized associated polynomial by 
    \[
      f_S^{\mathrm{norm}}(Y) = \overline{b_d}(\zeta)^{-1} \sum_{j=0}^d \overline{b_j}(\zeta) Y^j.
    \]

    We now define the associated polynomial for a general line or line segment $S$. 
    Let $V$ be the set of integer lattice points lying on $S \cap \{ (x,y) \in \R^2 \mid x \ge 0, y \ge 0 \}$. 
    We denote by $P, Q$ the points of $V$ with minimal and maximal abscissa, respectively, and 
    set $S'$ to be the line segment between $P$ and $Q$. 
    We define the associated polynomials of $f$ and $S$ by $f_S = f_{S'}$ and $f^{\mathrm{norm}}_S = f^{\mathrm{norm}}_{S'}$. 

    \begin{remark}
      If $f = \varphi^m$ and the $\varphi$-polygon of $f$ is one-sided, then $f_S = f_S^{\mathrm{norm}}$. 
    \end{remark}

  \subsection{Works of Ore, Montes and Nart on Newton polygons}\label{sec:Ore Montes Nart}
    
    In this section, we review the works of Ore and of Montes--Nart on Newton polygons. 

    The following two theorems are well-known facts about Newton polygons.

    \begin{theorem}[Theorem of the product]\label{thm:theorem of product}
      Let $f_1, \dots ,f_g \in \mathcal{O}_K[X]$ be monic polynomials, and write 
      $f(X) = f_1(X) \cdots f_g(X)$. 
      Let $S_{i,1}, \dots S_{i,k_i}$ be the sides of the principal part of the $\varphi$-polygon 
      of $f_i$. We define $M=\{-H(S_{i,j})/E(S_{i,j}) \in \Q_{< 0} \mid 1 \le i \le g,1 \le j \le k_i\}$, 
      and let $M'$ be the set of slopes of the sides in the principal part of the $\varphi$-polygon of $f$. 
      Then, $M=M'$, and for any side $S$ of the $\varphi$-polygon of $f$ with slope $-m$, we have 
      \[
        E(S)=\sum_{\substack{i,j\\H(S_{i,j})/E(S_{i,j})=m}} E(S_{i,j}), H(S)=\sum_{\substack{i,j\\H(S_{i,j})/E(S_{i,j})=m}} H(S_{i,j}). 
      \]
      Moreover, for any $m \in \Q_{>0}$, let $S_m$ be a side of the $\varphi$-polygon of $f$ of slope $-m$, 
      and let $S_{i,m}$ be a side of the $\varphi$-polygon of $f_i$ of slope $-m$. 
      Then, 
      \[
        f_S^{\mathrm{norm}}(Y) = \sideset{}{'}\prod_{i}f_{S_{i,m}}^{\mathrm{norm}}(Y), 
      \]
      where the product runs over all $i$ such that the $\varphi$-polygon of $f_i$ has a side of slope $-m$. 
    \end{theorem}

    \begin{theorem}[Theorem of the polygon]\label{thm:theorem of the polygon}
      Let $f \in \mathcal{O}_K[X]$ be a monic polynomial whose reduction is a power of $\varphi$ and whose $\varphi$-polygon
      consists of the sides $S_1, \dots ,S_g$. Let the slope of $S_{i}$ be $-m_i$. 
      Then, there exist monic polynomials $f_1, \dots ,f_g \in \mathcal{O}_K[X]$ 
      satisfying the following conditions: 
      \begin{enumerate}
        \item $ f(X) = \displaystyle \prod_{i = 1}^g f_i(X) $. 

        \item The $\varphi$-polygon of $f_i$ is one-sided, and it has the same length and height as $S_{i}$. 

        \item Let $S_{i}$ also denote the $\varphi$-polygon of $f_i$, then $f_{S_{i}}^{\mathrm{norm}}(Y) = (f_i)_{S_{i}}^{\mathrm{norm}}(Y)$. 

        \item If $\theta \in \overline{\Q}_p$ is a root of $f_i(X)$, then $v_K(\varphi(\theta)) = m_i$. 
      \end{enumerate}
    \end{theorem}

    Let $f \in \mathcal{O}_K[X]$ be a monic polynomial whose $\varphi$-polygon is one-sided, and denote this polygon by $S$. 
    We say $f$ is \emph{regular} if the associated polynomial $f_S(Y)$ is separable. 
    For a regular polynomial, Ore showed that the ramification index and inertia degree of the field defined by the polynomial can be computed from the $\varphi$-polygon. 

    \begin{theorem}[Theorem of Ore]\label{thm:theorem of Ore}
      Let $f(X) \in \mathcal{O}_K[X]$ be a monic polynomial whose $\varphi$-polygon is one-sided, and denote this polygon by $S$. 
      Suppose that $\overline{f}(Y) = \overline{\varphi}(Y)^n (n \in \Z_{\ge 0})$. 
      We write $e=E(S)/d$ and $h=H(S)/d$. Let 
      \[
        f_S(Y) = \psi_1(Y)^{e_1} \cdots \psi_t(Y)^{e_t} 
      \]
      be the factorization of $f_S^{\mathrm{norm}}$ into a product of powers of distinct irreducible polynomials
      of $\field{F}_K[Y]$. 
      Then $f$ admits a factorization 
      \[
        f(X) = f_1(X) \cdots f_t(X), 
      \]
      where each $f_i(X)$ is a monic polynomial with a one-sided $\varphi$-polygon $S_{i}$ 
      of the same slope as $S$, and whose associated polynomial $(f_i)_{S_{i}}(Y)$ equals $\psi_i(Y)^{e_i}$. 
      Moreover, we suppose that $f(X)$ is regular (i.e., $e_1 = \cdots = e_t = 1$), and let $\theta$ be a root of $f_i(X)$ and $L = K(\theta)$. 
      Then, all $f_i$ are irreducible, the prime ideal of $L$ equals $(\varphi(\theta)^b/\pi^c)\mathcal{O}_L$ 
      where $b$ and $c$ are positive integers such that $bh - ce = 1$, and $\mathcal{O}_L$ is the ring of integers of $L$, and 
      \[
        e(L/K) = e, f(L/K) = l \cdot \deg \psi_i(Y). 
      \]
    \end{theorem}

    Now, we define the two indices $i_K(f)$ and $i_\varphi(f)$ of a polynomial. 

    \begin{definition}\label{def:index of K}
      For a monic polynomial $f(X) \in \mathcal{O}_K[X]$, we define $i_K(f)$ as follows: 
      \begin{itemize}
        \item If $f(X)$ is irreducible, choose a root $\theta$, and put $L=K(\theta)$. 
          We then define 
          \[
            i_K(f) = v_K([\mathcal{O}_L:\mathcal{O}_K[\theta]])/[K:\Q_p]. 
          \]
        \item For monic irreducible polynomials $g(X),h(X) \in \mathcal{O}_K[X]$, 
          we set 
          \[
             r(g,h) = v_K(\mathrm{Res}(g,h)), 
          \]
          where $\mathrm{Res}(g,h)$ is the resultant of $g(X)$ and $h(X)$. 
        \item Let $f(X)$ be a product of monic irreducible polynomials $f_1(X), \dots ,f_t(X) \in \mathcal{O}_K[X]$. 
          We define 
          \[
            i_K(f) = \sum_{i = 1}^t i_K(f_i) + \sum_{i<j} r(f_i,f_j). 
          \]
      \end{itemize}
    \end{definition}

    \begin{definition}\label{def:index of phi}
      Let $f(X) \in \mathcal{O}_K[X]$ be a monic polynomial, and suppose that the principal part 
      of the $\varphi$-polygon of $f(X)$ consists of the sides $S_1, \dots ,S_r$. 
      For each $i$, denote the length, height, and degree of $S_{i}$ by $E_i$, $H_i$, and $d_i$, respectively. 
      Then, we define 
      \[
        i_{\varphi}(f) = \deg\varphi\cdot\left(\sum_{i = 1}^{r - 1} E_i \cdot 
        \left(\sum_{j = i + 1}^r H_j\right) + \frac{1}{2} \sum_{i = 1}^r 
        (H_i E_i - H_i - E_i + d_i)\right).
      \]
    \end{definition}

    The quantity $i_{\varphi}(f)/\deg\varphi(X)$ represents the number of points of 
    integer coordinates below or on the $\varphi$-polygon of $f$, excluding those lying on both axes. 
    See \cref{fig:1}.

    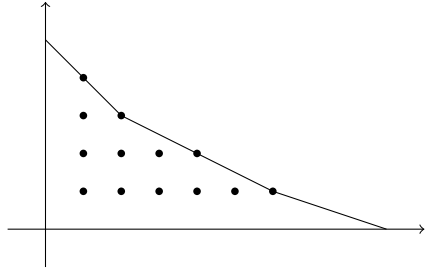
\begin{figure}
      \centering
      \begin{tikzpicture}
        \draw[->] (0,-0.5) -- (0,3);
        \draw[->] (-0.5,0) -- (5,0);
        \draw (0,2.5) -- (1,1.5) -- (3,0.5) -- (4.5,0);
        \foreach \x in {0.5,1,1.5,2}
          \fill[black] (0.5,\x) circle [radius=0.05];
        \foreach \x in {0.5,1,1.5}
          \fill[black] (1,\x) circle [radius=0.05];
        \foreach \x in {0.5,1}
          \fill[black] (1.5,\x) circle [radius=0.05];
        \foreach \x in {0.5,1}
          \fill[black] (2,\x) circle [radius=0.05];
        \foreach \x in {0.5}
          \fill[black] (2.5,\x) circle [radius=0.05];
        \foreach \x in {0.5}
          \fill[black] (3,\x) circle [radius=0.05];
      \end{tikzpicture}
      \caption{The number of points of integer coordinates}
      \label{fig:1}
    \end{figure}

    Let $f(X) \in \mathcal{O}_K[X]$ be a monic irreducible polynomial. 
    By Hensel's lemma, there exists a monic irreducible polynomial $\overline{\varphi}(Y) \in \field{F}_K[Y]$ 
    such that $\overline{f}(Y) = \overline{\varphi}(Y)^n (n \in \Z_{> 0})$. 
    Assume that $f(X)$ has a one-sided $\varphi$-polygon $S$, and set $e=E(S)/d$, $h=H(S)/d$. 
    By Theorem \ref{thm:theorem of Ore}, there exists a monic irreducible polynomial $\psi(Y) \in \field{F}_K[Y]$ 
    such that $f_S(Y) = \psi(Y)^a (a \in \Z_{> 0})$. Let $\theta$ be a root of $f(X)$ and put $L=K(\theta)$. 
    We define $\gamma_i=\varphi(\theta)^i/\pi^{[iH/E]} \in \mathcal{O}_L \quad (i = 0, \dots ,n - 1)$. 

    \begin{proposition}\label{prop:prop of Montes Nart}\cite[Proposition]{MR1152908}
      Under the above assumptions, the following conditions are equivalent: 
      \begin{enumerate}
        \item $i_K(f) = i_{\varphi}(f).$

        \item $e(L/K) = ea$, $f(L/K) = l \cdot \deg\psi(Y)$ and 
        [$a = 1$ or $\mathfrak{p}_L = \varphi(\gamma_e)\mathcal{O}_L$]. 
      \end{enumerate}
    \end{proposition}

    To prove the above proposition, Montes and Nart proved the following lemma on the valuation of the resultant. 
    
    \begin{lemma}\label{lem:valuation of resultant}\cite[Lemma 2]{MR1152908}
      Let $f,g \in \mathcal{O}_K[X]$ be monic polynomials of degrees $n$, $n'$, respectively. 
      Assume that the $\varphi$-polygons $S$ and $S'$ of $f$ and $g$ are one-sided, and denote their heights by $H$ and $H'$, respectively. 
      Then, 
      \[
        v_K(\mathrm{Res}(f,g)) \ge \min\{nH',n'H\}. 
      \]
      Moreover, equality holds if and only if either slope of $S$, $S'$ are different or $\mathrm{Res}(f_S,g_{S'}) \ne 0$. 
    \end{lemma}

     Using \cref{prop:prop of Montes Nart}, Montes and Nart proved the following theorem, which is a generalization of theorem of Ore.

    \begin{theorem}\label{thm:theorem of Montes Nart}\cite[Theorem 1]{MR1152908}
      Under the assumptions and notation of Theorem \ref{thm:theorem of Ore}, suppose that $i_K(f) = i_{\varphi}(f)$. 
      Then, $f_1(X), \dots ,f_t(X)$ are irreducible, and 
      \[
        e(L/K) = ee_i, f(L/K) = l\cdot\deg\psi_i(Y). 
      \]
    \end{theorem}

    We can check the condition $i_K(f) = i_{\varphi}(f)$ as follows.

    We continue to use the above notation. 
    Let $T$ be the unique unramified extension of $K$ of degree $l$. 
    Let $\mathcal{O}_K[X,Y] \to \field{F}_K[X,Y]/(\bar{\varphi}(X)) \cong \field{F}_T[Y]$ be the natural homomorphism, 
    and choose a lift $\Psi_i(X,Y)$ of $\psi_i(Y)$ such that 
    \[
      \Psi_i(X,Y) = \sum_j a_j(X)Y^j,\quad \deg a_j(X) < \deg\varphi(X) \quad \text{or} \quad a_j(X) = 0. 
    \]
    We define 
    \[
      f^0(X) = \pi^H \prod_{i = 1}^t \Psi_i(X,\varphi(X)^e/\pi^h)^{e_i} \in \mathcal{O}_K[X], 
    \]
    and 
    \[
      f^1(X) = f(X) - f^0(X). 
    \]
    Let $\tilde{S}$ be the side $S$ shifted upwards by $1/e$. 
    Then, the $\varphi$-polygon of $f^1(X)$ has no points below $\tilde{S}$. 

    \begin{theorem}\label{thm:equivalent condition of Montes Nart}\cite[Criterion 1]{MR1152908}
      The equality $i_K(f) = i_{\varphi}(f)$ holds if and only if we have either 
      $a_i = 1$ or $\psi_i(Y) \nmid f^1_{\tilde{S}}(Y)$ in $\F_T[Y]$ for all $i$. 
    \end{theorem}

    Let $f(X) \in \mathcal{O}_K[X]$ be a monic polynomial, not necessarily one-sided.
    Let $S_1, \dots ,S_r$ be the sides of the principal part of the $\varphi$-polygon of $f$. 
    By \cref{thm:theorem of the polygon}, $f(X)$ admits a factorization 
    $f(X)=f_1(X) \cdots f_r(X)$ where the $\varphi$-polygon of $f_i$ is $S_{i}$. 
    Since each $f_i$ is one-sided, we can define $f_i^0$ and $f_i^1$ as above. 
    We then define 
    \[
      f^1(X)=f(X)-LC_{\varphi}(f) \cdot \prod_{i=1}^r f_i^0(X), 
    \]
    where $LC_{\varphi}(f)$ is the leading coefficient of the $\varphi$-expansion of $f$. 
    Now, we define $\widetilde{S_{i}}$ again.
    Let $\overline{S}_{i}$ be a line containing $S_{i}$.
    Put $\widetilde{S_{i}}$ be the minimal line segment containing $\overline{S}_{i} \cap (\Z_{\ge 0} \times \Z_{\ge 0})$. 

    \begin{theorem}\label{thm:criterion 2 of Montes Nart}\cite[Criterion 2]{MR1152908}
      We use the above notation. Let $\psi(Y)$ be an irreducible factor of $(f_i)_{S_{i}}$. 
      Then, $\psi(Y)$ divides $(f_i^1)_{\widetilde{S_{i}}}(Y)$ if and only if it divides $(f^1)_{\widetilde{S_{i}}}(Y)$. 
    \end{theorem}

    To apply \cref{thm:theorem of Montes Nart} to the many-sided case, 
    we have to check the condition on indices. 

    \begin{definition}
        Let $f(X) \in \Oint_{K}[X]$.
        We call a set $\{\varphi_{i}(X)\}_{i \in I}$ of monic irreducible polynomials in $\Oint_{K}[X]$
        a \emph{set of irreducible pullbacks} of $f$
        if $\{\overline{\varphi_{i}}(X)\}_{i \in I}$ coincides with 
        the set of the distinct irreducible factors of $\overline{f}(X) \in \F_{K}[X]$.
    \end{definition} 

    The following theorem is an analogue of \cref{thm:equivalent condition of Montes Nart}. 

    \begin{theorem}[{\cite[Theorem 2]{MR1152908} \cite[Theorem 3.1]{index}}]\label{thm:theorem of index}
        Let $\Phi$ be a set of irreducible pullbacks of $f$.
        Then,
        \[
            \iK(f) \ge \sum_{\varphi \in \Phi} \iphi(f). \label{1:ineq:thm}
        \]
        Moreover, the equality holds if and only if for any $\varphi(X) \in \Phi$
        we have either $\Psi(Y)^2 \nmid f_{S}(Y)$ or $\Psi(Y) \nmid f^{1}_{\tilde{S}}(Y)$
        for any side $S$ of the principal parts
        of the $\varphi$-polygons and
        for any irreducible factor $\Psi(Y)$ of $f_{S}(Y)$.
    \end{theorem}

  \subsection{Works of Kölle and Schmid}\label{sec:Kölle Schmid}

    Now, we review the theorem of Kölle and Schmid. 
    Throughout this section, we suppose that $\overline{f}(X) = X^n$ and $\varphi(X) = X$. 

    In this section, we assume that $K$ is a (global) number field. 
    Let $f(X)=\sum_{i = 0}^n a_i X^{n - i} \in \mathcal{O}_K[X]$ be a polynomial, $L$ be the splitting field of $f(X)$, and $G = \Gal_K(f) = \Gal(L/K)$. 
    Let $\mathfrak{p}$ be a prime ideal of $K$. 
    We denote the set of all roots of $f$ by $Z_f$. Then $G$ acts on $Z_f$. 
    Let $\mathfrak{P}$ be a prime ideal of $L$ lying above $\mathfrak{p}$, and $G_\mathfrak{P}$ and $I_\mathfrak{P}$ be the 
    decomposition group and the inertia group of $L/K$ at $\mathfrak{P}$, respectively. 
    Then, $G_{\mathfrak{P}} = \Gal(L_{\mathfrak{P}}/K_{\mathfrak{p}})$. 
    Denote the set of roots of $f$ whose $\mathfrak{p}$-adic valuation is equal to $-m$ by $Z_{f,m}$. 
    Then $G_{\mathfrak{P}}$ acts on $Z_{f,m}$. 
    We set $G_{\mathfrak{P}}^{Z_{f,m}} = G_{\mathfrak{P}}/\mathrm{Stab}_{G_{\mathfrak{P}}}(Z_{f,m})$. 
    $G_{\mathfrak{P}}^{Z_{f,m}}$ also acts on $Z_{f,m}$. 
    Let $S_m$ be the side of $\varphi$-polygon $S$ of $f(X)$ with slope $m$. 
    Put $e=E(S_m)/d$ and $h=H(S_m)/d$, and let $(s,v_{\mathfrak{p}}(a_s))$ be its initial point. 
    We define $f_m =  \displaystyle{\sum_{(i,v_{\mathfrak{p}}(a_i)) \in S_m}}  (-1)^i a_s^{-1} a_{i}X^{de-i}$. 
    Let the distinct irreducible factors of $f_S^{\mathrm{norm}}$ over $\F_{K_{\mathfrak{p}}}$ have 
    degrees $d_1, \dots ,d_r$ (so that $\sum_{i = 1}^r d_i = d = \deg (f_S^{\mathrm{norm}})$). 
    
    The following is the theorem of Kölle and Schmid. 

    \begin{theorem}[{\cite[Theorem]{MR2102807},\cite{MR2557856}}]\label{thm:theorem of Kölle Schmid}
      We use the above notation, and suppose that $f$ is regular. 
      Then, $f_m$ is separable and $|Z_{f,m}| = de$. 
      If in addition $\mathfrak{p} \nmid e$, then the following hold:
      \begin{enumerate}
        \item For every root $\beta$ of $f_m$, there exists $\theta \in Z_{f,m}$ such that $K_{\mathfrak{p}}(\beta) = K_{\mathfrak{p}}(\theta)$, and vice versa. 
        \item $G_{\mathfrak{P}}^{Z_{f,m}}$ is permutation isomorphic to $\Gal_{K_\mathfrak{p}}(f_m)$. 
        \item $I_{\mathfrak{P}}^{Z_{f,m}}$ is cyclic and generated by an element which is the product of $d$ disjoint $e$-cycles on $Z_{f,m}$. 
        \item $G_{\mathfrak{P}}^{Z_{f,m}} = \langle \sigma,\tau \rangle$ has just $r$ orbits of sizes $d_1 e, \dots ,d_r e$ 
          and is a metacyclic group, with $\sigma^{-1}\tau\sigma = \tau^{N\mathfrak{p}}$. 
          Let $\mu = \mathrm{lcm}(\omega,d_1, \dots ,d_r) \quad (\omega = \mathrm{ord}_e N\mathfrak{p})$, the order of the (cyclic) group $G_{\mathfrak{P}}^{Z_{f,m}}/I_{\mathfrak{P}}^{Z_{f,m}}$ 
          is divisible by $\mu$, and it is a divisor of $\mu e$. 
          This order is equal to $\mu$ if $e = 1$ and $d = 1$, and if $r = 1$ and $\gcd (\omega,d) = 1$. 
      \end{enumerate}
    \end{theorem}

\section{Main theorem}\label{sec:Main theorem}

  \subsection{One-sided case}\label{sec:One-sided}

  In this section, we prove a generalization of the theorem of Kölle and Schmid for polynomials having a one-sided $\varphi$-polygon. 
  Throughout this section, we assume that $\overline{f}(X) = X^n$ and $\varphi(X) = X$. 

  To state the theorem, we define the following polynomial. 
  We use the notation of \cref{sec:Prelim}. 
  Let $f(X) \in \mathcal{O}_K[X]$ be a monic polynomial having a one-sided polygon $S$. 
  We define $\tilde{f}^1(X) \in \mathcal{O}_K[X]$ by $\tilde{f}^1(X) = (f^1)^0(X)$, 
  using the definition of \cref{sec:Ore Montes Nart} applied to $f^1$ and its $\varphi$-polygon $\tilde{S}$. 
  We put $\tilde{f}=f^0 + \tilde{f}^1$. 

  In the following theorem, we use the notation of \cref{sec:Kölle Schmid}. 

  \begin{theorem}\label{thm:our main theorem}
      Suppose that $f$ has a one-sided $\varphi$-polygon and $i_{K_{\mathfrak{p}}}(f) = i_{\varphi}(f)$. 
      Then, $\tilde{f}$ is separable, $f$ has $de$ roots, and the $\mathfrak{p}$-adic valuations of all roots of $f$ are equal to $m$. 
      Moreover, let $f_S = \psi_1^{a_1} \cdots \psi_r^{a_r}$ be the irreducible factorization over $\field{F}_{K_{\mathfrak{p}}}$, 
      and put $a = \mathrm{lcm}\{a_1, \dots ,a_r\}$. 
      If in addition $\mathfrak{p} \nmid ea$, then the following hold:
      \begin{enumerate}
        \item For every root $\beta$ of $\tilde{f}$, there exists $\theta \in Z_f$ such that $K_{\mathfrak{p}}(\beta) = K_{\mathfrak{p}}(\theta)$, and vice versa. 
        \item $G_{\mathfrak{P}}$ is permutation isomorphic to $\Gal_{K_\mathfrak{p}}(\tilde{f})$. 
        \item $I_{\mathfrak{P}}$ is cyclic and generated by an element which is the product of $d$ disjoint $ea$-cycles on $Z_f$. 
        \item $G_{\mathfrak{P}} = \langle \sigma,\tau \rangle$ has just $r$ orbits of sizes $d_1 ea_1, \dots ,d_r ea_r$ 
          and is a metacyclic group, with $\sigma^{-1}\tau\sigma = \tau^{N\mathfrak{p}}$. 
          Let $\mu = \mathrm{lcm}(\omega,d_1, \dots ,d_r) \quad (\omega = \mathrm{ord}_{ea} N\mathfrak{p})$, the order of the (cyclic) group $G_{\mathfrak{P}}/I_{\mathfrak{P}}$ 
          is divisible by $\mu$, and it is a divisor of $\mu ea$. 
          % This order is equal to $\mu$ if $ea = 1$ or $d = 1$, and if $r = 1$ and $\gcd (\omega,d) = 1$. 
      \end{enumerate}
    \end{theorem}

  To prove this theorem, we prove several propositions on polynomials having a one-sided $\varphi$-polygon. 
  Throughout this section, we assume that $i_K(f) = i_{\varphi}(f)$. 
  Since all points of $\varphi$-polygons of $f(X)$ and $\tilde{f}(X)$ lying on $S$ coincide, 
  we have $f_S(Y) = \tilde{f}_S(Y)$. 

  \begin{proposition}\label{prop:claim 2}
    $\tilde{f}(X)$ is separable. 
  \end{proposition}

  \begin{proof}
    Note that $f_S(Y) = \tilde{f}_S(Y)$ and ${f^1}_{\tilde{S}}(Y) = {\tilde{f}^1}_{\tilde{S}}(Y)$. 
    By \cref{thm:theorem of index}, 
    \[
      i_K(f) = i_{\varphi}(f) \Leftrightarrow i_K(\tilde{f}) = i_{\varphi}(\tilde{f}). 
    \]
    Therefore, by \cref{thm:theorem of Montes Nart}, $\tilde{f}(X)$ factors into distinct irreducible polynomials. 
  \end{proof}

  Let $F_f$ and $F_{\tilde{f}}$ be the sets of irreducible factors of $f$ and $\tilde{f}$, respectively. 
  The following proposition follows from \cref{thm:theorem of Montes Nart}. 

  \begin{proposition}\label{prop:claim 3}
    There exists a one-to-one correspondence $\phi \leftrightarrow \psi$ between $F_f$ and $F_{\tilde{f}}$ satisfying the following conditions: 
    \begin{enumerate}
      \item There exists an irreducible polynomial $g(Y) \in \field{F}_K[Y]$ such that $\phi_{S_1}(Y) = \psi_{S_2}(Y) = g(Y)^a$, where $S_1$ and $S_2$ are the $\varphi$-polygons of $\phi$ and $\psi$, respectively. 
      \item Let $\theta$ and $\beta$ be roots of $\phi$ and $\psi$, respectively. Then, 
      \begin{align}
        e(K(\theta)/K) &= e(K(\beta)/K) = ea, \\
        f(K(\theta)/K) &= f(K(\beta)/K) = \deg(g). 
      \end{align}
    \end{enumerate}
  \end{proposition}

  \begin{proposition}\label{prop:claim 4}
    Suppose that $f_S = \psi^a$ for some irreducible polynomial $\psi \in \field{F}_K[Y]$, and 
    that $\mathfrak{p} \nmid ea$. 
    Then there exists a one-to-one correspondence $\theta \leftrightarrow \beta$ between the roots of $f$ and those of $\tilde{f}$ such that $K(\theta) = K(\beta)$. 
  \end{proposition}

  To prove this proposition, we prove the following lemma. 

  \begin{lemma}\label{lem:classification of roots}
    Let $\{\alpha_1, \dots ,\alpha_t\}$ be the set of roots of $f_S(Y)$ in $\overline{\field{F}}_K$ 
    with multiplicities $a_1, \dots ,a_t$, respectively. Assume that $\mathfrak{p} \nmid e$. 
    Let $Z_f$ be the set of roots of $f$ in $\overline{K}$. 
    We define an equivalence relation on $Z_f$ by $v_K(\beta - \beta') > \frac{h}{e} \quad (\beta, \beta' \in Z_f)$ and 
    denote the decomposition of $Z_f$ into equivalence classes by $Z_f = \coprod_{i = 1}^l Z_f^i$. 
    Then $l = et$, and $\#Z_f^i = a_i$. 
    In particular, if $f_S = \psi^a$ for some irreducible polynomial $\psi \in \field{F}_K[Y]$, 
    we have $\#Z_f^i = a$ for all $i$. 
  \end{lemma}

  \begin{proof}
    We first show that it is sufficient to consider the case that $f_S$ is a power of a polynomial of degree one. 
    Let $\hat{T}$ be the unique unramified extension of $K$ of degree $\deg f_S$, then 
    $f_S = \prod_{k = 1}^t {\psi_k}^{a_k}$ for some distinct linear polynomials $\psi_k \in \field{F}_{\hat{T}}[Y] (k = 1, \dots ,t)$. 
    By \cref{thm:theorem of Ore}, $f$ admits a factorization $f = \prod_{k = 1}^t f_k$ over $\hat{T}$ such that $(f_k)_{S_{k}} = {\psi_k}^{a_k}$, where $S_{k}$ is the $\varphi$-polygon of $f_k$. 
    Denote the decomposition of $Z_{f_k}$ into equivalence classes by $Z_{f_k} = \coprod_{i = 1}^{l_k} Z_{f_k}^i$. 
    We will show that $Z_{f_k}^i$ forms an equivalence class in $Z_f$ for all $k$ and $i$. 
    It suffices to show that the elements of $Z_{f_k}$ and $Z_{f_{k'}}$ are not equivalent for $k \ne k'$. 
    Let $\beta \in Z_{f_k}$ and $\beta' \in Z_{f_{k'}}$, and  
    let $g_k$ and $g_{k'}$ be the minimal polynomials of $\beta$ and $\beta'$, respectively. 
    By \cref{thm:theorem of product}, the $\varphi$-polygons of $g_k$ and $g_{k'}$ have the same slope as $S$, 
    and $(g_k)_{S_{k}'} = \psi_k^{b_k}$ and $(g_{k'})_{S_{k'}'} = \psi_{k'}^{b_{k'}}$ for some $b_k$ and $b_{k'}$, where $S_{k}'$ and $S_{k'}'$ are the $\varphi$-polygons of $g_k$ and $g_{k'}$, respectively. 
    Let $\alpha_k$ and $\alpha_{k'}$ be roots of $g_k^0$ and $g_{k'}^0$, respectively. 
    Then, $\overline{\pi^{-h}{\alpha_k}^e}, \overline{\pi^{-h}{\alpha_{k'}}^e} \in \field{F}_{\hat{T}}$ are roots of $\psi_k$ and $\psi_{k'}$, respectively. 
    Since $\psi_k$ and $\psi_{k'}$ are coprime, we have $\overline{\pi^{-h}{\alpha_k}^e} \ne \overline{\pi^{-h}{\alpha_{k'}}^e}$ in $\field{F}_{\hat{T}}$, and 
    \[
      v_K(\pi^{-h}{\alpha_k}^e - \pi^{-h}{\alpha_{k'}}^e) = 0. 
    \]
    This implies that 
    \[
      v_K(\pi^{-\frac{h}{e}}\alpha_k - \pi^{-\frac{h}{e}}\alpha_{k'}) = 0, 
    \]
    and we have 
    \[
      v_K(\alpha_k - \alpha_{k'}) = \frac{h}{e}. 
    \]
    By \cref{lem:valuation of resultant}, 
    \[
      v_K(\mathrm{Res}(g_k,g_k^0)) > b_k e \cdot b_k h. 
    \]
    Let $L_0$ be the compositum of the splitting fields of $g_k$ and $g_k^0$ over $\hat{T}$. 
    Since the Galois group $\Gal(L_0/\hat{T})$ acts transitively on $Z_{g_k}$,
    for any two roots $\beta_0, \beta_1 \in Z_{g_k}$, there exists $\sigma_0 \in \Gal(L_0/\hat{T})$ such that 
    $\sigma_0(\beta_0) = \beta_1$, and we have $v_K(\prod_{\alpha \in Z_{g_k^0}} (\beta_0 - \alpha)) = v_K(\sigma_0(\prod_{\alpha \in Z_{g_k^0}} (\beta_0 - \alpha))) = v_K(\prod_{\alpha \in Z_{g_k^0}} (\beta_1 - \alpha))$. 
    Hence, 
    \[
      v_K(\prod_{\alpha \in Z_{g_k^0}} (\beta - \alpha)) > b_k h. 
    \]
    Moreover, by \cref{thm:theorem of the polygon}, for any $\alpha \in Z_{g_k^0}$, we have 
    \begin{align}
      v_K(\beta - \alpha) &\ge \min\{v_K(\beta),v_K(\alpha)\} \\
      &=\frac{h}{e}. 
    \end{align}
    These inequalities show that there exists a root $\alpha_k$ of $g_k^0$ such that 
    \[
      v_K(\beta - \alpha_k) > \frac{h}{e}. 
    \]
    On the other hand, since
    \[
      \frac{h}{e} = v_K(\alpha_k - \alpha_{k'}) = v_K(\alpha_k - \beta + \beta - \alpha_{k'}), 
    \]
    we have 
    \[
      v_K(\beta - \alpha_{k'}) = \frac{h}{e}. 
    \]
    Similarly, the same holds for $\beta'$. Therefore, we have 
    \[
      v_K(\beta - \beta') = v_K(\beta - \alpha_k + \alpha_k - \beta') = \frac{h}{e}. 
    \]

    Therefore, we may assume that $f_S = \psi^a$ for some linear polynomial $\psi$. 
    In this case, $f^0(X) = (X^e - b)^a \quad (b \in \mathcal{O}_K)$. 
    Let $\alpha$ be a root of $f^0$, and $\zeta$ a primitive $e$-th root of unity. 
    Then, the set of roots of $f^0$ is $Z_{f^0} = \{\alpha,\alpha\zeta, \dots ,\alpha\zeta^{e-1}\}$. 
    By \cref{lem:valuation of resultant}, for $\beta \in Z_f$, there exists an $i \in \{0, \dots ,e-1\}$ such that 
    \[
      v_K(\beta - \alpha\zeta^i) > \frac{h}{e}. 
    \]
    Moreover, since $\mathfrak{p} \nmid e$, we have 
    \[
      v_K(\prod_{i \ne 0} (\alpha - \alpha\zeta^i)) = v_K(\alpha^{e - 1}\prod_{i \ne 0} (1 - \zeta^i)) = (e - 1) \cdot \frac{h}{e}. 
    \]
    For $i \ne 0$, we have 
    \[
      v_K(\alpha - \alpha\zeta^i) \ge \min\{v_K(\alpha),v_K(\alpha\zeta^i)\} = \frac{h}{e}. 
    \]
    Hence, 
    \[
      v_K(\alpha - \alpha\zeta^i) = \frac{h}{e} \label{formula:lem-1}. 
    \]

    Put $Z_f^i = \{\beta \in Z_f \mid v_K(\beta - \alpha\zeta^i) > \frac{h}{e}\} \quad (i = 0, \dots ,e - 1)$. 
    We will show that each $Z_f^i$ forms an equivalence class in $Z_f$. 
    For $\beta, \beta' \in Z_f^i$, we have 
    \begin{align}
      v_K(\beta - \beta') &= v_K(\beta - \alpha\zeta^i + \alpha\zeta^i - \beta') \\
      &= \min\{v_K(\beta - \alpha\zeta^i),v_K(\beta' - \alpha\zeta^i)\} \\
      &> \frac{h}{e}. 
    \end{align}
    Hence, $\beta$ and $\beta'$ are equivalent. 
    Let $i \ne j$, and let $\beta \in Z_f^i$ and $\beta' \in Z_f^j$. 
    Then, we have 
    \begin{align}
      v_K(\beta - \beta') &= v_K(\beta - \alpha\zeta^j + \alpha\zeta^j - \beta') \\
      &\ge \min\{v_K(\beta - \alpha\zeta^j),v_K(\beta' - \alpha\zeta^j)\}. 
    \end{align}
    By (\ref{formula:lem-1}), 
    \begin{align}
      v_K(\beta - \alpha\zeta^j) &= v_K(\beta - \alpha\zeta^i + \alpha\zeta^i - \alpha\zeta^j) \\
      &= \frac{h}{e}. 
    \end{align}
    This implies that 
    \[
      v_K(\beta - \beta') = \frac{h}{e}. 
    \]
    Therefore, $Z_f^i$ forms an equivalence class. 

    It remains to show that $\#Z^i_f = a$ for all $i$. 
    Let $T = \hat{T}(\zeta)$. Since $\mathfrak{p} \nmid e$, $T/K$ is unramified. 
    If $e = 1$, there is only one equivalence class, and the claim is clear. 
    Assume that $e \ne 1$. 
    Let $\beta_i \in Z_f^i$, and let $L$ and $L_0$ be the splitting fields of $f$ and $f^0$ over $T$, respectively. 
    Since $v_K(\alpha) = \frac{h}{e}$ and $e$ and $h$ are coprime, we have $v_K(\alpha^m) \notin \Z$ for each $m = 1, \dots ,e-1$, hence $\alpha^m \notin T$. 
    By Kummer theory, $f^0$ is irreducible over $T$, and for any $i$ and $j$, there exists $\sigma_{ij} \in \Gal(LL_0/T)$ such that $\sigma_{ij}(\alpha \zeta^i) = \alpha \zeta^j$. 
    Then, we have 
    \[
      v_K(\beta_i - \alpha\zeta^i) = v_K(\sigma_{ij}(\beta_i) - \alpha\zeta^j), 
    \]
    hence $\sigma_{ij}(\beta_i) \in Z^j_f$. This implies that $\#Z^j_f \ge \#Z^i_f$ for any $i$ and $j$, and so $\#Z^j_f = \#Z^i_f$. 
    Therefore, we have shown that $\#Z^i_f = a$. 
  \end{proof}
  
  \begin{remark}
    In the proof of this lemma, the assumption $i_K(f) = i_{\varphi}(f)$ is not required. 
  \end{remark}

  \begin{proof}[Proof of \cref{prop:claim 4}]
    Put $n = \deg \psi$. Then, the length, height, and degree of $S$ are given by 
    $E = \deg f = \deg \tilde{f} = nea, H = nah, d = na$. 

    First, we show that it is sufficient to consider the case that $n = 1$. 
    Let $\hat{T}$ be the unramified extension of $K$ of degree $n$. 
    Then, $\psi = \prod_{i = 1}^n \psi_i$ for some distinct linear polynomials $\psi_i \quad (i = 1, \dots ,n)$. 
    Let $\theta$ and $\beta$ be roots of $f$ and $\tilde{f}$, respectively. 
    By \cref{thm:theorem of Montes Nart}, $\hat{T} \subseteq K(\theta), \hat{T} \subseteq K(\beta)$. 
    Hence, we have $K(\theta) = \hat{T}(\theta), K(\beta) = \hat{T}(\beta)$. 
    By \cref{prop:prop of Montes Nart}, $\tilde{\psi}(\pi^{-h}\theta^e)$ is a uniformizer of $K(\theta)$, where $\tilde{\psi}$ is a lift of $\psi$. 
    Let $g \in \mathcal{O}_{\hat{T}}[X]$ be the minimal polynomial of $\theta$ over $\hat{T}$. 
    By \cref{thm:theorem of Ore}, $g_S = \psi_i^l$ for some $i$ and $l$. 
    Assume that $i = 1$. 
    Since $g^0(\theta) + g^1(\theta) = g(\theta) = 0$, we have 
    \[
      v_K(g^0(\theta)) = v_K(g^1(\theta)) \ge H + \frac{1}{e}. 
    \]
    Moreover, let $\tilde{\psi}_1$ be a lift of $\psi_1$, since $g^0(\theta) = \pi^H\tilde{\psi}_1(\pi^{-h}\theta^e)^l$, 
    we have 
    \[
      v_K(\tilde{\psi}_1(\pi^{-h}\theta^e)) \ge \frac{1}{le} \ge \frac{1}{ea}. 
    \]
    On the other hand, we have 
    \[
      v_K(\tilde{\psi}_1(\pi^{-h}\theta^e)) \le v_K(\tilde{\psi}(\pi^{-h}\theta^e)) \le \frac{1}{ea}. 
    \]
    It follows that all the inequalities become equalities, $l = a$, and $\tilde{\psi}_1(\pi^{-h}\theta^e)$ is a uniformizer of $\hat{T}(\theta)$. 
    By \cref{prop:prop of Montes Nart}, the equality $i_{\hat{T}}(g) = i_{\varphi}(g)$ holds. 
    Therefore, it is sufficient to prove the proposition for $\hat{T}$ and $g$. 

    Hence, we can suppose that $\deg \psi = 1$. Let $\theta$ be a root of $f$. 
    Let $f(X) - \tilde{f}(X) = \sum_{i = 0}^E c_{E - i}X^i$. 
    Since $f_S = \tilde{f}_S, f_{\tilde{S}}^1 = \tilde{f}_{\tilde{S}}^1$, 
    for each $i$, we have 
    \[
      v_K(c_{E - i}) \ge H - \frac{h}{e} \cdot i + \frac{2}{e}. 
    \]
    Thus, we have 
    \[
      v_K(\tilde{f}(\theta)) = v_K(f(\theta) - \tilde{f}(\theta)) \ge H + \frac{2}{e} > H + \frac{1}{e}. 
    \]
    By \cref{thm:theorem of Montes Nart}, the ramification index of $K(\theta)/K$ is equal to $ea$. 
    Denote the set of roots of $\tilde{f}$ by $\{\beta_j \mid j = 1, \dots ,ea\}$. 
    For all $j$, we have 
    \[
      v_K(\theta - \beta_j) \ge \min\{v_K(\theta),v_K(\beta_j)\} = \frac{h}{e}. 
    \]
    If $v_K(\theta - \beta_{j_1}) > \frac{h}{e}$ and $v_K(\beta_{j_1} - \beta_{j_2}) = \frac{h}{e}$, we have 
    \[
      v_K(\theta - \beta_{j_2}) = v_K(\theta - \beta_{j_1} + \beta_{j_1} - \beta_{j_2}) = \frac{h}{e}. 
    \]
    By \cref{lem:classification of roots}, there are at most $a$ roots such that $v_K(\theta - \beta_j) \ge \frac{h}{e} + \frac{1}{ea}$. 
    Noting that $\tilde{f}$ has $ea$ roots, we obtain from above inequalities that there exists $j_0 \in \{1, \dots E\}$ such that 
    \[
      v_K(\theta - \beta_{j_0}) > \frac{h}{e} + \frac{1}{ea} \label{formula:prop4-2}. 
    \]
    Assume that $j_0 = 1$, and denote $\beta = \beta_1$. 

    We consider the derivative $\tilde{f}'$ of $\tilde{f}$. 
    Since $\tilde{f}^0(X) = \pi^H\tilde{\psi}(\pi^{-h}X^e)^a$, we have 
    \[
      (\tilde{f}^0)'(X) = ea\pi^{(a-1)h}X^{e - 1}\tilde{\psi}(\pi^{-h}X^e)^{a - 1}. 
    \]
    Noting that $\mathfrak{p} \nmid ea$ and $\tilde{\psi}(\pi^{-h}\beta^e)$ is a uniformizer of $K(\beta)$, it follows that 
    \[
      v_K((\tilde{f}^0)'(\beta)) = ah - \frac{h}{e} +\frac{a - 1}{ea}. 
    \]
    Since the points of  the $\varphi$-polygon of $\tilde{f}^1$ lie on $\tilde{S}$, we have 
    \[
      v_K((\tilde{f}^1)'(\beta)) \ge ah - \frac{h}{e} + \frac{1}{e} > v_K((\tilde{f}^0)'(\beta)). 
    \]
    Hence, 
    \begin{align}
      \sum_{i \ne 1} v_K(\beta - \beta_i) &= v_K(\tilde{f}'(\beta)) \\ 
      &=v_K((\tilde{f}^0)'(\beta)) \\
      &= ah - \frac{h}{e} + \frac{a - 1}{ea}
    \end{align}
    Since $v_K(\beta - \beta_i) \ge \frac{h}{e}$ for $i \ne 1$, we have 
    \[
      v_K(\beta - \beta_i) \le \frac{h}{e} + \frac{a - 1}{ea}. 
    \]
    By \cref{lem:classification of roots}, there are just $a - 1$ roots $\beta_i$ satisfying $v_K(\beta - \beta_i) > \frac{h}{e}$. 
    Consequently, for all $i = 2, \dots ,ea$, 
    \[
      v_K(\beta - \beta_i) \le \frac{h}{e} + \frac{1}{ea}. \label{formula:prop4-3}
    \]

    From (\ref{formula:prop4-2}) and (\ref{formula:prop4-3}), for all $i \ne 1$, 
    we have 
    \[
      v_K(\theta - \beta) > v_K(\beta - \beta_i). 
    \]
    By Krasner's lemma, this gives $K(\beta) \subset K(\theta)$. 
    We have equality since $f$ and $\tilde{f}$ are irreducible over $K$ of the same degree. 

    It remains to show the uniqueness of $\beta$. 
    If there exists another root $\beta'$ of $\tilde{f}$ satisfying (\ref{formula:prop4-2}), then we have 
    \[
      v_K(\beta - \beta') > \frac{h}{e} + \frac{1}{ea}. 
    \]
    This contradicts (\ref{formula:prop4-3}). 
  \end{proof}

  \begin{proof}[Proof of \cref{thm:our main theorem}]
  We now prove \cref{thm:our main theorem}. 
  We use the notation of \cref{sec:Kölle Schmid}. 

  By \cref{prop:claim 2}, $\tilde{f}$ is separable. 
  By \cref{thm:theorem of the polygon}, the $\mathfrak{p}$-adic valuations of all $de$ roots of $f$ are equal to $m$. 

  Assume that $\mathfrak{p} \nmid ea$. 
  (i) follows from \cref{prop:claim 3} and \cref{prop:claim 4}. 
  By (i), the splitting fields of $f$ and $\tilde{f}$ coincide, say $L$.
  Hence, $G_{\mathfrak{P}}$ is isomorphic to $\Gal_{K_{\mathfrak{p}}}(\tilde{f})$ as a finite group. 
  Let $\theta$ be a root of $f$, and $g(X) \in K_{\mathfrak{p}}[X]$ the minimal polynomial of $\theta$ over $K_{\mathfrak{p}}$. 
  Let $\tilde{g}$ be the irreducible factor of $\tilde{f}$ corresponding to $g$ under the bijection of \cref{prop:claim 3}, and 
  $\beta$ be the root of $\tilde{g}$ corresponding to $\theta$ under the bijection of \cref{prop:claim 4}. 
  Since $\deg g = \deg \tilde{g}$, the $G_{\mathfrak{P}}$-orbit of $\theta$ has the same length as the $\Gal_{K_{\mathfrak{p}}}(\tilde{f})$-orbit of $\beta$. 
  Moreover, the stabilizers of $\theta$ and $\beta$ can be written as $\Gal(L/K_{\mathfrak{p}}(\theta))$ and $\Gal(L/K_{\mathfrak{p}}(\beta))$, respectively. 
  Since $K_{\mathfrak{p}}(\theta) = K_{\mathfrak{p}}(\beta)$, it follows that the stabilizers of $\theta$ and $\beta$ coincide. 
  These imply (ii). 

  Put $G' = \Gal(L/K_{\mathfrak{p}})$. By (ii), we may identify $G_{\mathfrak{P}}$ with $G'$, together with their action on $Z_f$ (or $Z_{\tilde{f}}$). 
  Let $\hat{T}$ be the maximal unramified subextension of $L/K_{\mathfrak{p}}$, then $I_{\mathfrak{P}} = \Gal(L/\hat{T})$. 

  By \cref{thm:theorem of Montes Nart}, the factorization of \cref{thm:theorem of Ore} 
  \[
    \tilde{f}(X) = \tilde{f}_1(X) \cdots \tilde{f}_r(X)
  \]
  is an irreducible factorization. Hence, $Z_f$ has $r$ orbits of lengths $d_1 ea_1, \dots ,d_r ea_r$ under the action of $G$. 
  Moreover, since $L/\hat{T}$ is tamely ramified, $G_{\mathfrak{P}}$ is a metacyclic group. 
  Let $\tau$ be a generator of $I_{\mathfrak{P}}$, and $\sigma$ be a lift of a generator of $G_{\mathfrak{P}}/I_{\mathfrak{P}}$. 
  Then we have $\sigma^{-1}\tau\sigma = \tau^{N\mathfrak{p}}$. 

  Let $\varepsilon \in \bar{K}_{\mathfrak{p}}$ be a primitive $ea$-th root of unity. 
  Then $[K_{\mathfrak{p}}(\varepsilon):K_{\mathfrak{p}}] = \omega$. 
  Let $\beta$ be a root of $\tilde{f}$, and $T'$ be the maximal unramified subextension of $K_{\mathfrak{p}}(\beta)/K_{\mathfrak{p}}$. 
  Since $K_{\mathfrak{p}}(\beta)/K_{\mathfrak{p}}$ is tamely ramified and by the theorem of Pauli and Roblot (\cite[Theorem 7.2]{MR1836924}), 
  there exists an element $u \in T'$ such that $K_{\mathfrak{p}}(\beta)$ coincides with the extension over $T'$ defined by the polynomial $X^{ea_i} - u$. 
  Let $\alpha$ be a root of $X^{ea_i} - u$. Then the set of roots of $X^{ea_i} - u$ can be written as $\{\alpha,\alpha\varepsilon_i,\dots,\alpha\varepsilon_i^{ea_i-1}\}$, 
  where $\varepsilon_i$ is a primitive $ea_i$-th root of unity. 
  Since the normal closure of $K_{\mathfrak{p}}(\beta)$ is contained in $L$, $\varepsilon_i \in L$ for all $i=1, \dots ,r$. 
  Hence, $\varepsilon \in L$. 
  Since $\mathfrak{p} \nmid ea$, $K_{\mathfrak{p}}(\varepsilon)/K_{\mathfrak{p}}$ is unramified. Hence, $K_{\mathfrak{p}}(\varepsilon) \subset \hat{T}$. 
  It follows that $\omega$ divides $|G_{\mathfrak{P}}/I_{\mathfrak{P}}|$. 
  Moreover, let $\beta_i$ be a root of $\tilde{f}_i$. 
  By \cref{thm:theorem of Montes Nart}, the degree over $K_{\mathfrak{p}}$ of the maximal unramified subextension of $K_{\mathfrak{p}}(\beta_i)/K_{\mathfrak{p}}$ 
  is equal to $d_i$. It follows that $d_i$ divides $|G_{\mathfrak{P}}/I_{\mathfrak{P}}|$ for each $i$. 
  Therefore $\mu$ divides $|G_{\mathfrak{P}}/I_{\mathfrak{P}}|$. 

  Since $d_i$ divides $|G_{\mathfrak{P}}/I_{\mathfrak{P}}|$ for each $i$, $\tilde{f}_S$ factors as a product of linear polynomials over $\field{F}_{\hat{T}}$. 
  Since $\hat{T}/K_{\mathfrak{p}}$ is unramified, the $\varphi$-polygon of $\tilde{f}$ over $\hat{T}$ coincides with that over $K_{\mathfrak{p}}$. 
  This implies that the equality $i_{\hat{T}}(\tilde{f}) = i_{\varphi}(\tilde{f})$ also holds. 
  Indeed, the polynomial $\tilde{f}^1$ defined over $\hat{T}$ coincides with that defined over $K_{\mathfrak{p}}$. 
  Noting that $\tilde{f}^1_{\tilde{S}}$ is a polynomial over $\field{F}_{K_{\mathfrak{p}}}$, it follows from \cref{thm:equivalent condition of Montes Nart} that 
  $i_{\hat{T}}(\tilde{f}) = i_{\varphi}(\tilde{f})$. 
  By \cref{thm:theorem of Montes Nart}, $\tilde{f}$ factors as a product of irreducible polynomials of degree $ea_i$. 
  Thus, $I_{\mathfrak{P}}$-orbits of $Z_{\tilde{f}}$ have lengths $ea_i$. 
  Since $L/K_{\mathfrak{p}}$ is tamely ramified, $I_{\mathfrak{P}}$ is a cyclic group. 
  These imply (iii). 

  There exists a unique subfield $T$ of $\hat{T}$ such that $[T:K_{\mathfrak{p}}] = \mu$. 
  Let $\beta$ be a root of $\tilde{f}$, and suppose that $\tilde{f}_j$ is the minimal polynomial of $\beta$. 
  Since $d_j \mid \mu$, $(\tilde{f}_j)_S$ factors as a product of linear polynomials over $\field{F}_T$. 
  By \cref{thm:theorem of Montes Nart}, $T(\beta)/T$ is a totally ramified extension of degree $ea_j$. 
  Since $\varepsilon \in T$, $T(\beta)/T$ is a Galois extension by the theorem of Pauli and Roblot and Kummer theory. 
  On the other hand, $\hat{T}/T$ is unramified, and hence $T(\beta)$ and $\hat{T}$ are linearly disjoint over $T$. 
  It follows that $T(\beta)\hat{T}/T(\beta)$ is a cyclic extension of degree $[\hat{T}:T]$. 
  Noting that $L$ is the compositum of all these $T(\beta)$ and $T(\beta)/T$ is a Galois extension of degree $ea_i$, 
  the exponent of its Galois group $\Gal(L/T)$ is $ea$. 
  Hence, the exponent of $\Gal(T(\beta)\hat{T}/T)$ divides $ea$. 
  Since $\Gal(T(\beta)\hat{T}/T(\beta))$ is a cyclic subgroup of $\Gal(T(\beta)\hat{T}/T)$, 
  its degree $[\hat{T}:T]$ divides $ea$. Therefore $|G_{\mathfrak{P}}/I_{\mathfrak{P}}|$ divides $\mu ea$. 
  Thus, (iv) has been proved. 

  \end{proof}

  \subsection{Many-sided case}\label{sec:many-sided}

  In this section, we study polynomials whose Newton polygons are not necessarily one-sided.
  Montes and Nart showed that,
  if the decomposition of $f(X)$ corresponding to its Newton polygon is $f_{1}(X)\cdots f_{r}(X)$,
  then one can obtain from $f(X)$ the associated polynomial of each $f_{i}$, as well as a polynomial containing information on $(f_{i}^{1})_{\widetilde{S_{i}}}$ (\cref{thm:criterion 2 of Montes Nart}).
  For their purposes, the polynomial obtained from $f(X)$ is sufficient, although it may contain extraneous terms. 
  For our purposes, however, we need to recover the exact polynomial $(f_i^1)_{\widetilde{S_{i}}}$. 
  In this section, we give a method to recover it from $f(X)$. 

  Let the sides of the Newton polygon of $f(X)$ be denoted from left to right by $S_{1},\dots,S_{r}$,
  and let the decomposition according to \cref{thm:theorem of the polygon} be denoted by
  \[
    f(X)=f_{1}(X)\cdots f_{r}(X).
  \]
  Since each Newton polygon of $f_{i}(X)$ has only one side as mentioned above in \cref{thm:criterion 2 of Montes Nart},
  we can define $f_{i}^{1}(X)$ for $i=1,\dots,r$.
  Thus we have
  \begin{align}
    f(X) &= \prod_{i=1}^{r}(f_{i}^{0}(X)+f_{i}^{1}(X)) \\
    &= \prod_{i=1}^{r}f_{i}^{0}(X) + \sum_{i=1}^{r}\left(f_{i}^{1}(X)\prod_{j\ne i}f_{j}^{0}(X)\right) + \sum_{i<j}\left(f_{i}^{1}(X)f_{j}^{1}(X)\prod_{k\ne i,j}f_{k}^{0}(X)\right)
  \end{align}
  By definition, we obtain
  \[
    f^{1}(X) = \sum_{i=1}^{r}\left(f_{i}^{1}(X)\prod_{j\ne i}f_{j}^{0}(X)\right) + \sum_{i<j}\left(f_{i}^{1}(X)f_{j}^{1}(X)\prod_{k\ne i,j}f_{k}^{0}(X)\right).
  \]

  Let $\ell\in\{1,2,\dots,r\}$.
  \begin{itemize}
    \item   If $\ell = i$ or $\ell = j$, then  the Newton polygon of $f_{i}^{1}(X)f_{j}^{1}(X)\prod_{k\ne i,j}f_{k}^{0}(X)$ does not intersect with $\widetilde{S_{\ell}}$.
    \item   If $\ell \ne i,j$, then $f_{i}^{1}(X)f_{j}^{1}(X)\prod_{k\ne i,j}f_{k}^{0}(X)$ is divisible by $f_{\ell}^{0}(X)$.
    \item   If $\ell \ne i$, then $f_{i}^{1}(X)\prod_{j\ne i}f_{j}^{0}(X)$ is divisible by $f_{\ell}^{0}(X)$.
  \end{itemize}
  Therefore we have
  \[
    (f^{1})_{\widetilde{S_{\ell}}}(Y) \equiv \left(f_{\ell}^{1}\prod_{i\ne\ell}f_{i}^0\right)_{\widetilde{S_{\ell}}}(Y) \mod (f_{\ell}^{0})_{\widetilde{S_{\ell}}}(Y).
  \]
  Since the Newton polygon of $\prod_{i\ne\ell}f_{i}^{0}(X)$ has no side with the same slope as $S_{\ell}$,
  $\left(\prod_{i\ne\ell}f_{i}^{0}\right)_{\widetilde{S_{\ell}}}(Y)$ is a monomial.
  Here, $\widetilde{S_{\ell}}$ is obtained by translating $S_{\ell}$ so that it is tangent to the Newton polygon of $\prod_{i\ne\ell}f_{i}^{0}(X)$.
  Put 
  \[
    A_{\ell} = \sum_{i<\ell}\lfloor \deg f_{i}^{0}(X)/e_{\ell}\rfloor. 
  \]
  We define $R_{\ell}(Y)$ to be a representative of the residue class of $(f^{1})_{\widetilde{S_{\ell}}}(Y)$ modulo $(f_{\ell}^{0})_{\widetilde{S_{\ell}}}(Y)$ of the form 
  \[
    R_{\ell}(Y) = Y^{A_{\ell}}Q_{\ell}(Y). 
  \]
  We choose such a representative so that 
  \[
    \deg Q_{\ell}(Y) < \deg (f_{\ell}^0)_{\widetilde{S_{\ell}}}(Y). 
  \]
  Under this condition, $R_{\ell}(Y)$ is uniquely determined. 
  
  Now we obtain the following theorem.
  \begin{theorem}\label{thm:our main theorem 2}
    With the above notation, we have 
    \[
      R_{\ell}(Y) = \left(\prod_{i\ne\ell}f_{i}^{0}\right)_{\widetilde{S_{\ell}}}(Y) \cdot (f_{\ell}^{1})_{\widetilde{S_{\ell}}}(Y).
    \]
  \end{theorem}

  \begin{proof}
  By the preceding discussion, the two sides agree up to the coefficient coming from the monomial
  \[
    \left(\prod_{i\ne\ell}f_{i}^{0}\right)_{\widetilde{S_{\ell}}}(Y). 
  \]
  To complete the proof, it remains to compare this coefficient. 

  The coefficient is given by the product of the constant terms of the polynomials corresponding to the sides 
  lying to the right of $S_{\ell}$. 
  Equivalently, it is the term corresponding to the point at which the Newton polygon of $\prod_{i\ne\ell}f_{i}(X)$ 
  is tangent to $\widetilde{S_{\ell}}$. 
  In $\field{F}_K[Y]$, this coefficient is equal to 
  \[
    \prod_{i > \ell} f_i^0(0). 
  \]
  This proves the theorem. 
  \end{proof}

  \section{An Example}\label{sec:example}

  Consider a polynomial 
  \[
    f(X) = X^5 - 43X^3 - 7X^2 + 354X + 1072.
  \]
  We will obtain the Galois group $G$ of $f$ over $\Q$.

  This polynomial has the discriminant
  \[
    \mathrm{Disc}(f) = 1677927225500100 = 2^{2} 3^{14} 5^{2} 1873^{2}.
  \]
  For $p=7$, we have
  \[
    f(X) \equiv X^5 + 6X^3 + 4X +1 \mod 7.
  \] 
  Since the polynomial is irreducible in $\F_{7}[X]$, we can apply Dedekind's theorem to $f$ and $7$,
  and the decomposition group of $7$ is isomorphic to the cyclic group $C_5$ of order $5$. 

  For $p=5$, since
  \begin{align}
    f(X-3)  &= X^5 - 15X^4 + 47X^3 + 110X^2 -360X + 865\\
            &= X^5 - 3\cdot 5X^4 + 47X^3 + 22\cdot 5X^2 - 72\cdot 5X + 173\cdot 5,
  \end{align}
  the $5$-adic Newton polygon of $f(X-3)$ has two sides, one of which is the non-principal part.
  The principal side has the length $3$ and the height $1$, hence the associated polynomial of $f$ and the side has degree one
  and we can apply the theorem of K\"{o}lle and Schmid.
  As a result, the decomposition group corresponding to the principal part is isomorphic to $C_{3}$.

  Finally, we consider $p=3$.
  Since
  \begin{align}
    f(X+1)  &= X^5 + 5X^4 - 33X^3 - 126X^2 + 216X + 1377 \\
            &= X^5 + 5X^4 - 11\cdot 3X^3 - 14\cdot 3^2X^2 + 8\cdot 3^3X + 17\cdot3^4,
  \end{align}
  the $3$-adic Newton polygon of $f(X+1)$ has two sides, one of which is the non-principal part.
  Let $S$ be the principal side of the polygon, $g$ the polynomial corresponding to $S$, and $h$ the polynomial corresponding to the non-principal part.
  Then, we have
  \begin{align}
    &f_{S}(Y) = 2(Y^2 + Y + 2)^2,\\
    &f^{1}_{\tilde{S}}(Y) = 2Y(Y+2)(Y^2+2Y+2).
  \end{align}
  $Y^2 + Y + 2$ does not divide $f^{1}_{\tilde{S}}(Y)$ in $\F_{3}[Y]$.
  Thus, we can apply our main theorem to this polynomial.

  To apply our theorem, we calculate $g^{1}_{\tilde{S}}$ as in \cref{sec:many-sided}.
  \begin{align}
    g^{1}_{\tilde{S}}(Y)  &= \left(f^{1}_{\tilde{S}}(Y) - f_{S}(Y)\right)/h_{0}(0)\\
                          &= (Y^3 + 2Y^2 + 1)/2\\
                          &= 2Y^3 + Y^2 + 2.
  \end{align}
  Thus, we obtain
  \[
    \tilde{g}(X) = X^4+(2\cdot 3+2\cdot3^2)X^3+(2\cdot 3^2+1\cdot 3^3)X^2+(1\cdot3^3)X+(1\cdot 3^4+2\cdot 3^5).
  \]
  This polynomial is irreducible over $\Q_{3}$, and its discriminant is a square.
  Since $f_{S}(Y) = 2(Y^2 + Y + 2)^2$, our theorem implies that the Galois group of $\tilde{g}$ over $\Q_{3}$ has order $4$.
  Hence, the Galois group is $C_{2}\times C_{2}$.

  Thus, $G$ contains subgroups isomorphic to $C_3$, $C_5$, and $C_2 \times C_2$. Hence $|G|$ is divisible by $3$, $5$, and $4$, and therefore by $60$. 
  Since the discriminant of $f$ is a square, we have $G \le A_5$. Thus $|G| = 60$, and hence $G$ is isomorphic to $A_5$. 

  \section*{Acknowledgements}
  The authors are grateful to Professor Masanari Kida for his support and encouragement.

\end{document}